\documentclass[reqno,a4paper,12pt]{amsart}
\usepackage{amsmath,amsthm,amsfonts,amssymb, amsmath, mathrsfs}
\usepackage{verbatim, graphicx, ifthen}
\usepackage[T1]{fontenc}

\allowdisplaybreaks

\newtheorem{theoa}{Theorem}

\newtheorem{theorem}{Theorem}
\newtheorem{lemma}{Lemma}
\newtheorem*{thm}{Theorem}

\newtheorem{corollary}{Corollary}
\theoremstyle{definition}

\theoremstyle{remark}

\newcommand{\e}{\mathrm{e}}
\newcommand{\im}{\mathrm{i}}
\newcommand{\dif}{\mathrm{d}}

\newcommand{\Z}{\mathbb{Z}}

\newcommand{\abs}[1]{|#1|}
\newcommand{\Abs}[1]{\left|#1\right|}
\newcommand{\norm}[1]{\|#1\|}
\newcommand{\Norm}[1]{\left\|#1\right\|}
\newcommand{\inner}[2]{\left\langle #1|#2 \right\rangle}
\newcommand{\seq}[1]{(#1)}
\newcommand{\Seq}[1]{\left(#1\right)}
\newcommand{\set}[1]{\big\{#1\big\}}
\newcommand{\Set}[1]{\left\{#1\right\}}

\newcommand{\loc}{\mathrm{loc}}
\newcommand{\N}{\mathbb{N}}
\newcommand{\R}{\mathbb{R}}
\newcommand{\T}{\mathbb{T}}
\newcommand{\C}{\mathbb{C}}
\newcommand{\Pri}{\mathbb{P}}

\newcommand{\Bigoh}[1]{\mathcal{O} \left( #1 \right)}
\newcommand{\littleoh}{o}
\newcommand{\Littleoh}[1]{o \left( #1 \right)}
\newcommand{\Hp}{\mathscr{H}}

\renewcommand{\Re}{\operatorname{Re}}

\title
{Modified zeta functions as kernels of integral operators}

\author{Jan-Fredrik Olsen}
\address{Department of Mathematical Sciences, Norwegian University of
Science and Technology (NTNU), NO-7491 Trondheim, Norway}
\thanks{The author is supported by the Research Council of Norway grant 160192/V30. }

\begin{document}

\subjclass[2000]{30B50 (primary), 11M45, 47G10 (secondary)}

\begin{abstract}
	The modified zeta functions $\sum_{n \in K} n^{-s}$, where $K \subset \N$, converge absolutely for $\Re s > 1/2$. These generalise the Riemann zeta function which is known to have a meromorphic continuation to all of $\C$ with a single pole at $s=1$. 
	Our main result is a characterisation of the modified zeta functions that have pole-like behaviour at this point.
	This behaviour is defined
	by considering the modified zeta functions as kernels of certain integral operators on the spaces $L^2(I)$ for symmetric and bounded intervals $I \subset \R$. 
	We also consider the special case when the set $K \subset \N$ is assumed to have arithmetic structure. In particular, we look at local $L^p$ integrability properties of the modified zeta functions on the abscissa $\Re s=1$ for $p \in [1,\infty]$.
\end{abstract}

\maketitle

\section{Introduction}

We consider the behaviour of the modified zeta functions defined by
\begin{equation}  \label{definition: k-zeta function}
	\zeta_K(s) = \sum_{n \in K} \frac{1}{n^{s}}, \quad K \subset \N,
\end{equation}
near the point $s=1$. Here $s = \sigma + \im t$ denotes the complex variable. The infinite series defining these functions converge absolutely in the half-plane $\sigma > 1$. We refer to these as $K$-zeta functions. Note that for $K = \N$, the formula \eqref{definition: k-zeta function} defines the Riemann zeta function. 

The main objective of this paper is, for general $K \subset \N$, to find an operator-theoretic generalisation of the classical result, due to B.~Riemann, that the Riemann zeta function can be expressed as
\begin{equation} \label{asymptotic formula}
	 \zeta(s) = \frac{1}{s-1} + \psi(s),
\end{equation}
where $\psi$ is an entire function. (See \cite{edwards74} for an extensive discussion, as well as an English translation, of Riemann's original paper.)

To motivate this approach, we note that although N.~Kurokawa \cite{kurokawa87} found sufficient conditions on the sets $K$ for $\zeta_K$ to have an analytic continuation across the abscissa $\sigma = 1$, it was shown by
J.-P.~Kahane and H.~Queffelec \cite{kahane73, queffelec80} that for most choices of the subset $K$, in the sense of Baire categories, the $K$-zeta functions have the abscissa $\sigma=1$ as a natural boundary. So, instead of looking at the formula \eqref{asymptotic formula} as a statement about analytic continuation, we consider it as saying that the local behaviour of $\zeta(s)$ at $s=1$ is an analytic, and therefore small, perturbation of a pole with residue one. 

To interpret this in operator-theoretic terms, we define, for $K \subset \N$ and intervals $I$ of the form $(-T,T)$ with finite $T>0$, the family of operators
\begin{equation*} 
	\mathcal{Z}_{K,I}: g \in L^2(I) \\ \longmapsto \lim_{\delta\rightarrow 0}\frac{\chi_{I}(t)}{\pi} \int_I  g(\tau) \Re \zeta_K(1 + \delta + \im (t-\tau))  \dif \tau \in L^2(I).
\end{equation*}
Here  the characteristic function $\chi_{I}$ is applied to  emphasise that we look at $L^2(I)$ as a subspace of $L^2(\R)$.
To understand these operators, we consider the example $K=\N$. The formula \eqref{asymptotic formula} implies
\begin{equation*} 
	 \Re \zeta_\N ( 1+ \delta + \im t ) = \frac{\delta}{\delta^2 + t^2} + \Re \psi (1 + \delta + \im t),
\end{equation*}
whence
\begin{equation} \label{operator asymptotic}
	 \mathcal{Z}_{\N,I} = \mathrm{Id} + \Psi_{\N,I},
\end{equation}
for a compact operator $\Psi_{\N,I}$ and the identity operator $\mathrm{Id}$.
Indeed, the term $\pi^{-1}\delta/(\delta^2 + t^2)$ is the Poisson kernel which, under convolution, gives rise to the identity operator, while convolution with continuous kernels give compact operators (see Lemma \ref{convolution lemma}). Hence, $\mathcal{Z}_{\N,I}$ is a compact, and therefore a small perturbation of the identity operator.

In Theorem \ref{formula theorem}, we generalise the above formula in the following manner. We show that given $K\subset \N$, and a bounded and symmetric interval $I \subset \R$, there exist a subset $L \subset \R$ and a compact operator $\Phi_{K,I}$ such that
\begin{equation*}
	 \mathcal{Z}_{K,I} = \chi_I \mathcal{F}^{-1} \chi_L \mathcal{F} +  \Phi_{K,I}.
\end{equation*}
We remark that, intuitively, large $K \subset \N$ should correspond to large $L \subset \R$. In fact, it follows from our construction (see \eqref{definisjon av L} below) that if $K = \N$ then
$L=\R$. Hence, $\chi_I \mathcal{F}^{-1} \chi_L \mathcal{F} = \mathrm{Id}$ and we obtain again the formula \eqref{operator asymptotic}.

For general $K$ we present two additional results.
In Theorem \ref{lower bound theorem}, we characterise for which $K$ the operator $Z_{K,I}$ is bounded below.
By a stability theorem of semi-Fredholm theory, it turns out that this takes place if and only if the principal term of \eqref{generalised asymptotic formula}, the operator $\chi_I \mathcal{F}^{-1} \chi_L \mathcal{F}$, is bounded below. So, heuristically, a lower norm bound may be thought of as detecting the presence of a mass in the kernel of the integral operator $\mathcal{Z}_{K,I}$. 
Along with a result by B.~Panejah \cite{panejah66}, this enables us to show that the operator $\mathcal{Z}_{K,I}$ is bounded below in norm exactly for the sets $K$ for which there exists some $\delta \in (0,1)$ such that
\begin{equation} \label{intro density 1}
	\liminf_{x \rightarrow \infty} 	\frac{\pi_K(x) - \pi_K(\delta x)}{x} > 0,
\end{equation}
where $\pi_K(x)$ is the counting function of $K$. Note that this condition is independent of the interval $I$.
In Theorem \ref{identity theorem}, we obtain a complete characterisation of when $\mathcal{Z}_{K,I}$ satisfies a formula of the type \eqref{operator asymptotic}, in the sense that it is a compact perturbations of a scalar multiple of the identity operator; this happens if and only if the limit
\begin{equation} \label{intro density 2}
	 \lim_{x \rightarrow \infty} \frac{\pi_K(x)}{x}
\end{equation}
exists.
Theorem \ref{identity theorem} closely mirrors a recent generalisation, due to J.~Korevaar, of a classical tauberian result of S.~Ikehara \cite{ikehara31}.
Indeed, a formula of the type \eqref{operator asymptotic} holds if and only if a formula of the type \eqref{asymptotic formula} holds, where the appropriate substitute for $\psi$ extends to a nicely behaved distribution on the abscissa $\sigma=1$.

We also consider $K$-zeta functions for which $K \subset \N$ is assumed to have arithmetic structure. More specifically, we look at the case when $K$ consists exactly of the integers whose prime number decomposition contain only factors belonging to some specified subset $Q$ of the prime numbers $\Pri$. In this setting, the limit \eqref{intro density 2}  always exists, implying that theorems \ref{formula theorem} to \ref{identity theorem} are simplified. We state this as Theorem \ref{arithmetic theorem}. The condition \eqref{intro density 1} is now  expressed as
\begin{equation*}
	 \sum_{p \in \Pri \backslash Q} \frac{1}{p} < \infty.
\end{equation*}
Also, more detailed information about the $K$-zeta functions, in terms of certain $L^p$ estimates on the abscissa $\sigma=1$, is obtained. This is stated in Theorem   \ref{arithmetic theorem 2}. Finally, in Theorem \ref{beurling theorem}, we show that a prime number theorem for the subset of the primes $Q$, in a sense to be defined, neither implies, nor is implied by, a lower norm bound for the operators $Z_{K,I}$.

We make some remarks on related work. The operator $\mathcal{Z}_{\N,I}$ appears in the context of Hilbert spaces of Dirichlet series in a paper by J.-F. Olsen and E. Saksman \cite{olsen_saksman09} (see also section \ref{concluding remarks}). In work done by J.-P. Kahane \cite{kahane98}, a closely related functional is used to give a proof of the classical prime number theorem. Also, F.~Moricz \cite{moricz99} found precise estimates of the $L^p$ norms of a class of functions containing the $K$-zeta functions, along the segment $s \in (1,2)$ in terms of the counting functions $\pi_K$.

The structure of this paper is as follows. In sections \ref{section: generalisation of riemanns formula} and  \ref{section: lower norm bound} we state and prove, respectively, theorems \ref{formula theorem} and \ref{lower bound theorem}. In the latter section, we also make some comments relating our results to frame theory. Next, the Ikehara-Korevaar theorem is stated along with Theorem  \ref{identity theorem} and its proof in section \ref{section: korevaar}. 
In sections \ref{label: two arithmetic results} and \ref{section: prime numbers} we present results obtained under the additional assumption that $K$ has arithmetic structure. These are theorems \ref{arithmetic theorem} and \ref{arithmetic theorem 2}  and \ref{beurling theorem}.
Finally, in section \ref{concluding remarks}, we give some concluding remarks.


\section{A generalisation of Riemanns formula} \label{section: generalisation of riemanns formula}

In this section we give a generalisation of the formula \eqref{asymptotic formula}. 
Note that we use the convention
\begin{equation*}
	 \mathcal{F} : g \longmapsto \hat{g}(t) = \frac{1}{\sqrt{2\pi}} \int_\R g(\xi) \e^{-\im \xi t} \dif \xi
\end{equation*}
for the Fourier transform on $L^2(\R)$. Moreover, for $I \subset \R$ we identify the space $L^2(I)$ with the subspace of $L^2(\R)$ consisting of functions with support in $I$. In this way, we give sense to the expression $\mathcal{F} g$ for $g \in L^2(I)$.
\begin{theorem} \label{formula theorem}
	Let $K \subset \N$ be arbitrary, $I \subset \R$ be a bounded and symmetric interval, 
	and 
	\begin{equation} \label{definisjon av L}
		L = 	\bigcup_{n \in K} \bigg( \big(-\log (n+1), -\log n \big] \cup \big[\log n, \log (n+1)\big) \bigg).
	\end{equation}
	Then there exists a compact operator $\Phi_{K,I}$ such that
	\begin{equation} \label{generalised asymptotic formula}
	 	\mathcal{Z}_{K,I} =  \chi_I \mathcal{F}^{-1} \chi_L \mathcal{F} + \Phi_{K,I}.
	\end{equation}
\end{theorem}
\begin{proof}
The theorem is essentially the observation that by expanding the Dirichlet series of $2\Re  \zeta_K(s) = \zeta_K(s) + \overline{\zeta_K(s)}$, we get
\begin{equation} \label{z sum formula}
	 \mathcal{Z}_{K,I} g(t)  = \frac{\chi_I(t)}{\sqrt{2\pi}} \sum_{n \in K} \left(\frac{\hat{g}(\log n)}{n} n^{\im t} + \frac{\hat{g}(-\log n) }{n} n^{-\im t} \right).
\end{equation}
Indeed, the formula \eqref{generalised asymptotic formula}  follows if we show 
that for $g \in  \mathcal{C}^\infty_0(I)$ the difference
	\begin{equation*}
		\sum_{n \in K} \left( \frac{\hat{g}(\log n)}{n} n^{\im t} + \frac{\hat{g}(-\log n)}{n} n^{-\im t} \right) - \int_{L} \hat{g}(\xi)  \e^{- \im t \xi}\dif \xi
	\end{equation*}	
	is given by a compact operator $\Phi$. 
	Note that for $g \in \mathcal{C}^\infty_0(I)$ this sum converges absolutely since $\hat{g}(\xi) = \Bigoh{(1 + \xi^2)^{-1}}$. 
	In particular, since the Fourier transform is bounded, this implies that $\mathcal{Z}_{K,I}$ is bounded as an operator on $L^2(I)$. (This also follows from the fact that $\mathcal{Z}_{\N,I}$ is bounded and the remarks following Theorem \ref{lower bound theorem} relating these operators to frame theory.)

	Recall that 
	for $K\subset \N$,
	\begin{equation*} 
		L = 	\bigcup_{n \in K} \bigg( \big(-\log (n+1), -\log n \big] \cup \big[\log n, \log (n+1)\big) \bigg).
	\end{equation*}
	In order to simplify notation, we set
	$L_+ = L \cap (0,\infty)$ and consider the difference of only the positive frequencies,
	\begin{equation} \label{the positive frequencies}
		\sum_{n \in K} \frac{\hat{g}(\log n)}{n} n^{\im t}  - \int_{L_+} \hat{g}(\xi)  \e^{ \im t \xi} \dif \xi.
	\end{equation}
	It suffices to show that this is given by 
	a compact operator, say $2\pi \Phi_{+}$. 
	The same argument then works on the negative frequencies by taking complex conjugates, giving 
	us a compact operator $2\pi \Phi_-$. With the choice $\Phi = \Phi_+ + \Phi_-$, the proof is complete.
		
	By adding and subtracting intermediate terms, we see that the difference (\ref{the positive frequencies}) can be expressed as
	\begin{multline*}
		\underbrace{\sum_{n\in K} \frac{1}{n \log(1+\frac{1}{n})} \int_{L_n} \left( \hat{g}(\log n) n^{\im t} - \hat{g}(\xi) \e^{\im t \xi} \right) \dif \xi}_{(I)} \\ + \underbrace{\sum_{n\in K}  \left( \frac{1}{n \log(1+\frac{1}{n}) } -1 \right) \int_{L_n} \hat{g}(\xi) \e^{\im t \xi}  \dif \xi}_{(II)}.
	\end{multline*}
	
We want to interchange the integral and sum signs in these expressions. For $(I)$, it suffices to show that 
\begin{equation} \label{justify change}
	 \sum_{n \in K} \int_{L_n} \abs{\hat{g}(\log n) n^{\im t} - \hat{g}(\xi) \e^{\im \xi t}} \dif \xi \leq C \norm{g}_{L^2(I)}.
\end{equation}
for some constant $C>0$.
Note that by expressing the difference inside the absolute value as a definite integral, we have
\begin{equation*}
	\int_{L_n} \abs{n^{\im t} - \e^{\im t \xi}} \dif \xi \leq \abs{t} \frac{1}{n^2}.
	\end{equation*}
Pulling the absolute value sign inside of the expression for the Fourier transforms in combination with this inequality, gives us the bound
\begin{multline*}
	\int_{L_n} \abs{\hat{g}(\log n) n^{\im t} - \hat{g}(\xi) \e^{\im \xi t}} \dif \xi
	\leq
	\int_I \abs{g(\tau)} \int_{L_n} \abs{n^{\im (t-\tau)}- \e^{\im (t-\tau) \xi}}\dif \xi \dif \tau \\
	\leq 
	\frac{1}{n^2} \int_I \abs{t-\tau}\abs{g(\tau)} \dif \tau 
	\leq
	\frac{2\abs{I}}{n^2} \left(\int_I \abs{g(\tau)}^2 \dif \tau\right)^{1/2}.
\end{multline*}
Taking the sum, and using the Cauchy-Schwarz inequality, we get (\ref{justify change}) with constant $C = 2\abs{I} \zeta(4)^{1/2}$.
Interchanging the integral and sum signs, we get
\begin{equation*}
	(I) = \int_I g(\tau) \alpha(t-\tau) \dif \tau,
\end{equation*}
where
\begin{equation*}
	\alpha(\tau) = \frac{1}{\sqrt{2\pi}} \sum_{K} \frac{1}{n \log(1+1/n)} \int_{L_n} (n^{\im \tau} - \e^{\im \xi \tau}) \dif \xi.
\end{equation*}
By the same bound we used above, this sum converges absolutely and therefore the function $\alpha(t)$ is continuous on $I$. 
Similar arguments show that
\begin{equation*}
	(II) = \int_I g(\tau) \beta(t-\tau) \dif \tau,
\end{equation*}
where
\begin{equation*}
	\beta(\tau) = \sum_K \left(  \frac{1}{n \log(1+1/n)} - 1 \right) \int_{L_n} \e^{\im t \xi} \dif \xi
\end{equation*}
is a continuous function on $I$. Hence,
\begin{equation*}
	 2\pi \Phi_+ g(t) = \frac{1}{\sqrt{2\pi}} \int_I g(\tau) \Big(\alpha(t-\tau) + \beta(t-\tau)\Big) \dif \tau,
\end{equation*}
and so the compactness of $\Phi_+$ follows from Lemma \ref{convolution lemma}. By the comments of the first half of the proof this implies that $\Phi$ is also a compact operator.
\end{proof}

\section{Characterisation of $\mathcal{Z}_{K,I}$ which are bounded below in norm} \label{section: lower norm bound}
	
The following theorem explains when the operator $\mathcal{Z}_{K,I}$ is bounded below.
\begin{theorem} \label{lower bound theorem}
	Let $K \subset \N$ be arbitrary, $I \subset \R$ be a bounded and symmetric interval,
	and $L \subset \R$ be given by the relation (\ref{definisjon av L}).
	Then the followin{}g conditions are equivalent.
	\begin{gather}
		\mathcal{Z}_{K,I} \; \text{is bounded below on} \; L^2(I) \tag{a}\label{Z bounded below}  
		\\
		 \chi_I \mathcal{F}^{-1} \chi_L \mathcal{F} \; \text{is bounded below on} \; L^2(I) \tag{b} \label{fourier condition} 
		 \\
			\text{There exists} \; \delta \in (0,1) \; \text{such that} \;
			\liminf_{x \rightarrow \infty} \frac{\pi_K(x) - \pi_K(\delta x)}{x} > 0.\tag{c}\label{panejah's condition}
	\end{gather}
\end{theorem}

Before we give the proof, we mention a corollary of Theorem \ref{lower bound theorem}. Recall that a sequence of vectors $\seq{f_n}$ in some Hilbert space $H$ is called a frame if for all $f \in H$ there exists constants such that $\sum \abs{\inner{f}{f_n}}^2 \simeq \norm{f}^2$. It is a basic result of frame theory that such a sequence of vectors is a frame if and only if the operator defined by
\begin{equation*}
	 g \longmapsto \sum \inner{g}{f_n}f_n
\end{equation*}
is bounded and bounded below in norm (see e.g. \cite{christensen03}).
With this in mind, we define the sequence
\begin{equation} \label{the frame}
	\mathscr{G}_K = \Seq{ \ldots, \; \frac{3^{\im t}}{\sqrt{3}}, \; \frac{2^{\im t}}{\sqrt{2}}, \; 1, \; 1, \; \frac{2^{-\im t}}{\sqrt{2}}, \; \frac{3^{-\im t}}{\sqrt{3}}, \; \ldots },
\end{equation}
where $n$ is understood to run through $K \cup (-K)$. It is readily checked that equation \eqref{z sum formula} says exactly that $2\pi \mathcal{Z}_{K,I}$ is the frame operator of the sequence $\mathscr{G}_K$ when restricted to the space $L^2(I)$. Hence, we get the following.
\begin{corollary} \label{the corollary}
	Let $K \subset \N$ be arbitrary and $I \subset \R$ be a bounded and symmetric interval. Then the sequence of vectors $\mathscr{G}_K$ given by (\ref{the frame}), restricted to the interval $I$, forms a frame for $L^2(I)$ if and only if any of the conditions of Theorem \ref{lower bound theorem} holds.
\end{corollary}

	The proof of Theorem \ref{lower bound theorem} hinges in an essential way on the following lemma. 	
	\begin{lemma}[Second stability theorem of semi-Fredholm theory] \label{second stability theorem}
    	Let $X,Y$ be Banach spaces, let $Z : X \rightarrow Y$ 
		a continuous linear operator that is bounded below and $\Phi : X \rightarrow Y$
		be a compact operator. If $Z + \Phi$ is injective, then $Z + \Phi$ is bounded below.
	\end{lemma}
A proof of this lemma may be found in \cite[p. 238, Thm. 5.26]{kato66}.

\begin{proof}[Proof of Theorem \ref{lower bound theorem}]

	We proceed to show the equivalences $(b) \iff (c)$ and $(a) \iff (b)$.

$(b) \iff (c)$: Condition $(b)$ says that $\mathcal{F}^{-1} \chi_L \mathcal{F}$ is bounded below on $L^2(I)$. 
We begin by establishing that this is 
equivalent to $\chi_L \mathcal{F}$ being bounded below from $L^2(I)$ to $L^2(\R)$. Indeed,
	one direction is clear since
	\begin{equation*}
		\norm{\chi_L \mathcal{F} g}_{L^2(\R)} = \norm{\mathcal{F}^{-1}\chi_L \mathcal{F} g}_{L^2(\R)}
		\geq \norm{\chi_I \mathcal{F}^{-1}\chi_L \mathcal{F} g}_{L^2(I)}.
	\end{equation*}
		To prove the converse, assume that there exists some $\delta > 0$ such that for $g \in L^2(I)$
		\begin{equation} \label{hypothesis1}
	 		\norm{\chi_L \mathcal{F}g}_{L^2(\R)} \geq \delta \norm{g}_{L^2(I)}.
		\end{equation}
		Moreover, assume that for all $\epsilon > 0$ there exists an $g_\epsilon \in L^2(I)$ such that
		\begin{equation*}
	 		\norm{\chi_I\mathcal{F}^{-1} \chi_L \mathcal{F}g_\epsilon}_{L^2(I)} \leq \epsilon^2 \norm{g_\epsilon}_{L^2(I)}.
		\end{equation*}
		This implies that
		\begin{eqnarray*}
			\norm{\chi_I\mathcal{F}^{-1} \chi_{L^C} \mathcal{F}g_\epsilon }_{L^2(I)}
			&\geq&
			\norm{g_\epsilon }_{L^2(I)}
			-
			\norm{\chi_I \mathcal{F}^{-1} \chi_{L} \mathcal{F}g_\epsilon }_{L^2(I)} 
			\\
			&\geq&
			(1-\epsilon^2)\norm{g_\epsilon}_{L^2(I)}.
		\end{eqnarray*}
		On the other hand, the inequality (\ref{hypothesis1}) implies that
		\begin{eqnarray*}
			\norm{\chi_I\mathcal{F}^{-1}\chi_{L^C} \mathcal{F}g_\epsilon}_{L^2(I)}^2 
			&\leq& 
			\norm{ \mathcal{F}^{-1}\chi_{L^C} \mathcal{F}g_\epsilon}_{L^2(\R)}^2
			\\
			&=&
			\norm{g_\epsilon}^2_{L^2(I)} - \norm{\mathcal{F}^{-1} \chi_{L} \mathcal{F} g_\epsilon}_{L^2(\R)}^2 
			\\
			&\leq&
			(1-\delta^2) \norm{g_\epsilon}^2_{L^2(I)}.
		\end{eqnarray*}
		Combining these two inequalities, we find that $\epsilon \geq \delta$. 
		This leads to a contradiction since we may choose $\epsilon = \delta/2$.

We now invoke Panejah's theorem which says that the lower norm bound of $\chi_L \mathcal{F}$ on $L^2(\R)$ is equivalent to the condition that
 there exists a $\delta >0$ such that
	\begin{equation*}
		\inf_{\xi \in \R} \abs{L \cap (\xi-\delta, \xi)} > 0.
	\end{equation*}
	Finally, this is equivalent to
	\begin{equation*}
		\liminf_{\xi \rightarrow \infty} \frac{\pi_K(\e^{\xi-\delta},\e^{\xi})}{\e^\xi} > 0,
	\end{equation*}
	which is exactly condition $(c)$. Indeed, this
	 is just a matter of observing that
	\begin{equation*}
		  \frac{ \pi_K(\e^{\xi-\delta}, \e^{\xi})}{\e^\xi}
\leq \sum_{\log k \in (\xi-\delta,\xi)} \frac{1}{k} \leq   \e^{ \delta} \frac{ \pi_K(\e^{\xi-\delta}, \e^{\xi})}{\e^\xi}.
	\end{equation*}

$(a) \iff (b)$:
This equivalence follows essentially from the result from Lemma \ref{second stability theorem} and the identity 
$\mathcal{Z} =  \mathcal{F}^{-1} \chi_L \mathcal{F}  + \Phi$,
where $\Phi$ is a compact operator on $L^2(I)$ and $L$ is given by \eqref{definisjon av L}. What needs to be checked is that the lower bound of $\mathcal{Z}$ implies the injectivity of $\chi_I \mathcal{F}^{-1}\chi_L \mathcal{F}$, and vice versa.

By the equivalence of $(b)$ and $(c)$, which we just established, we know that if the operator $\mathcal{F}^{-1}\chi_L \mathcal{F}$ is bounded below, then there exists $\delta \in (0,1)$ such that
$\inf_{x \in \R} ( \pi_K(x) - \pi_K(\delta x))/x > 0$.
This is readily seen to imply that $\sum_{n \in K} n^{-1} = \infty$. We show that this is sufficient for the operator
$\mathcal{Z}$ to be injective. Indeed, define the operator 
\begin{equation*}
	 R : g \longmapsto \left( \ldots, \; \frac{\hat{g}(-\log 3)}{\sqrt{3}}, \;  \frac{\hat{g}(-\log 2)}{\sqrt{2}}, \; \hat{g}(0), \; \hat{g}(0), \;  \frac{\hat{g}(\log 2)}{\sqrt{2}}, \; \frac{\hat{g}(\log 3)}{\sqrt{3}}, \; \ldots \right).
\end{equation*}
By an easy computation we have $\mathcal{Z} = (2\pi)^{-1} R^\ast R$. Since an operator is always injective on the image of its adjoint it suffices to check that the hypothesis implies that $R$ is injective, i.e. that for $g \in L^2(I)$ then $\hat{g}(\pm \log n) = 0$ for all $n \in K$ implies $g = 0$.
To get a contradiction, assume that the function $f$ is non-zero. The function $\hat{g}$ is entire and of exponential type $\abs{I}/2$. In particular it is bounded on $\R$ and is therefore of the Cartwright class. A basic property of functions in this class (see \cite[lesson 17]{levin96}) is that the number of zeroes  with modulus less than $r>0$, which we denote by $\lambda(r)$, has to satisfy
\begin{equation*}
	\lim_{r \rightarrow \infty} \frac{\lambda(r)}{r} = \frac{\abs{I}}{\pi}.
\end{equation*}
	Let $\pi_K(x)$ be the counting function for $K$. Then $\lambda(r) \geq \pi_K(\e^r)$. The existence of the limit implies that $\pi_K(n) \leq C\log n$ for some $C>0$. Summing by parts and using this estimate, we see that
	\begin{equation} \label{in proof of z injectivity}
		\sum_{n \in K}^N \frac{1}{n} = \frac{\pi_K(N)}{N} + \sum_{n=1}^{N-1} \frac{\pi_K(n)}{n(n+1)}
		\leq 1 + C \sum_{n=1}^N \frac{\log n}{(n+1)^2},
	\end{equation} 
	which converges as $N \rightarrow + \infty$. Hence, we have a contradiction and so $g$ has to equal zero, as was to be shown.
We can now apply Lemma \ref{second stability theorem} to conclude that $\mathcal{Z}$ is bounded below on all of $L^2(I)$.

The same argument holds if we reverse the roles of $\mathcal{Z}$ and $\chi_I \mathcal{F}^{-1}\chi_L \mathcal{F}$ since the latter operator is injective whenever $K$ is non-empty. Indeed,
	assume that $K\neq \emptyset$ and let $g \in L^2(I)$ be such that $g \neq 0$. 
	It is clear that neither $\chi_L \mathcal{F} g$ nor $\mathcal{F}^{-1} \chi_L \mathcal{F} g$ can be equal to zero almost everywhere as functions in $L^2(\R)$. 
To conclude, we use the Plancherel-Parseval formula. For suppose that $\chi_I \mathcal{F}^{-1} \chi_L \mathcal{F}g = 0$. Since 
$g=\chi_I \mathcal{F}^{-1} \chi_{L^C} \mathcal{F}g + \chi_I \mathcal{F}^{-1}\chi_L \mathcal{F}g$,
this implies $\chi_{I} \mathcal{F}^{-1} \chi_{L^C} \mathcal{F}g = g$. And so 
	\begin{eqnarray*}
		\norm{g}_{L^2(I)}^2 
		&=& 
		\norm{\mathcal{F}^{-1} \chi_L \mathcal{F}g}_{L^2(\R)}^2
		+  \norm{  \mathcal{F}^{-1} \chi_{L^C} \mathcal{F}g}_{L^2(\R)}^2 \\		
		&\geq& 
		\norm{ \mathcal{F}^{-1}\chi_L \mathcal{F}g}_{L^2(\R)}^2
		+  \norm{ \chi_I \mathcal{F}^{-1}  \chi_{L^C} \mathcal{F}g}_{L^2(I)}^2 \\		
		&=& \norm{\mathcal{F}^{-1} \chi_L \mathcal{F}g}_{L^2(\R)}^2
		+ \norm{g}_{L^2(\R)}^2.
	\end{eqnarray*}
But from what is already established $\norm{ \mathcal{F}^{-1} \chi_L \mathcal{F}g}_{L^2(\R)} > 0$, which leads to a contradiction. This concludes the proof of the theorem.
\end{proof}

\section{Characterisation of $\mathcal{Z}_{K,I}$ which behave like the identity operator}  \label{section: korevaar}

The following result describes when $\mathcal{Z}_{K,I}$ is a compact perturbation of a scalar multiple of the identity operator. 
\begin{theorem} \label{identity theorem}
	Suppose $K \subset \N$ and $A \geq 0$. Then for all bounded and symmetric intervals $I \subset \R$, the operator defined by
	\begin{equation*}
		\Psi_{K,I} = \mathcal{Z}_{K,I} - A\mathrm{Id}
	\end{equation*}
	is compact 
	if and only if
	\begin{equation} \label{asymptotic density}
	 	\lim_{x \rightarrow \infty} \frac{\pi_K(x)}{x} =  A.
	\end{equation}
\end{theorem}
As mentioned in the introduction, this theorem should be compared to the following tauberian result due to S.~Ikehara \cite{ikehara31} and J.~Korevaar \cite{korevaar05}. Indeed, the sufficiency of the condition \eqref{asymptotic density} in Theorem \ref{identity theorem} follows directly from it.
Note that we call the distributional Fourier transform of $L^\infty$ functions which decay to zero at infinity pseudo-functions. These are in general distributions.
\begin{thm}[Ikehara 1931, Korevaar 2005] 
	Let $f(t)$ be a non-decreasing function with support in $(0,\infty)$, and suppose that
	the Laplace transform
	\begin{equation*}
		F(s) = \mathcal{L}f(s) = \int_0^\infty \frac{S(u)}{\e^u} \e^{-(s-1)u} \dif u
	\end{equation*}
	exists for $\sigma > 1$. For some constant $A$, let 
	\begin{equation*}
		g(s) = F(s) - \frac{A}{s-1}.
	\end{equation*}
	If $g(s)$ coincides with a pseudo-function on every bounded interval on the abscissa   $\sigma=1$ then
	\begin{equation*}
	 	\lim_{t\rightarrow \infty} \frac{S(u)}{\e^u} = A.
	\end{equation*}
	Conversely, if this limit holds, then $g$ extends to a pseudo-function on $\sigma=1$.
\end{thm}
We remark that it follows from the Ikehara-Korevaar theorem that $g$ extends to a pseudo-function on $\sigma=1$ if and only if $\e^{-u}{S(u)}$ tends to $A$.

The significance of pseudo-functions is that they are the class of distribtions which satisfy, by definition, the Riemann-Lebesgue lemma. In particular, this implies that the convolution-type operators they give rise to are compact operators. To make this more precise, we give the following lemma.
\begin{lemma} \label{convolution lemma} 
	Let $I \subset \R$ be a bounded and symmetric interval and $k \in L^1(2I)$. Then the operator defined by
	\begin{equation*} 
	 \Lambda : g \in L^2(I) \longmapsto   \chi_I \int_I g(\tau) k(t - \tau) \dif \tau \in L^2(I),
\end{equation*}
	is a compact operator on $L^2(I)$. 
	More generally, 
	if $\seq{k_{\delta}}_{\delta \in (0,1)}$ is a net of functions in $L^1_{\loc}(\R)$ converging in the sense of distributions to a pseudo-function $k$, then the operator
	\begin{equation*}
		\widetilde{\Lambda} : g \in L^2(I) \longmapsto  \lim_{\delta \rightarrow 0} 
		\chi_I \int_I g(\tau) k_\delta (t-\tau) \dif \tau
	\end{equation*}
	is bounded and compact on $L^2(I)$.
\end{lemma}
\begin{proof}
	Let $e_n(t)$ denote the Fourier characters of $L^2(2I)$, and let the Fourier expansion of $k$ on $L^2(2I)$ be given by
	\begin{equation*}
		k(t) = \sum_{n \in \Z} c_n e_n(t).
	\end{equation*}
	Hence, for $g \in L^2(I)$, 
	\begin{equation*} 
		\Lambda g(t) =  \abs{2I}^{1/2} \sum_{n \in \Z} c_n (g, e_n)_{L^2(I)} e_n(t).
	\end{equation*}
	By the Riemann-Lebesgue Lemma it follows that $\abs{c_n} \rightarrow 0$ as
	$\abs{n} \rightarrow \infty$ and the operator $\Lambda$ is seen to be compact.
	
	We turn to the second part of the statement.
	Let $g \in \mathcal{C}^\infty_0(I)$. Then
	\begin{align*}
			\lim_{\delta \rightarrow 0}\int_\R g(t-\tau) k_\delta (\tau) \dif \tau 
			&= \left( g(t - \cdot), k \right) \\
			&=  \int_\R  \hat{g}(\xi) \hat{k}(\xi)  \e^{\im t \xi} \dif \xi.
	\end{align*}
	By the dual expression of the $L^2(I)$ norm this is seen to be bounded by some constant times the $L^2$ norm of $g$. To see that it is
	compact,  define an operator on $\mathcal{C}^\infty_0(I)$ by 
	\begin{equation*}
		\Lambda_N g(t) =  \int_\R  \hat{g}(\xi) \hat{k}_N(\xi) \e^{\im t \xi}  \dif \xi, 
	\end{equation*}
	with $\hat{k}_N = \chi_N \hat{k}$. Since $\mathcal{F}\hat{k}_N \in L^1(2I)$ this is a compact operator by the first part of the lemma. 
	Moreover,
	\begin{equation*}
		 \norm{\widetilde{\Lambda} g - \Lambda_N g}_{L^2} \leq \norm{g}_{L^2}
		 \norm{\hat{k}}_{L^\infty(\abs{\xi} > N)}.
	\end{equation*}
 	Hence the sequence of compact operators $\Lambda_N$ approximates $\widetilde{\Lambda}$ in the uniform operator topology as $N \rightarrow \infty$.
\end{proof}

\begin{proof}[Proof of Theorem \ref{identity theorem}]
We use Lemma \ref{convolution lemma} to check the sufficiency of the condition \eqref{asymptotic density}.
Recall that $\pi_K(x)$ is the counting function of the integers $K$.
If we set $S(u) = \pi_K(\e^u)$ then 
Korevaar's result says that the density condition \eqref{asymptotic density}
implies that the function 
\begin{equation} \label{korevaar asymptotic formula} 
	  \psi_K(s) = \frac{1}{s} \zeta_K(s) - \frac{A}{s-1},
\end{equation}
extends to a pseudo-function on $\sigma=1$. In fact, it is straight-forward to check this implication directly since
\begin{equation*}
	\frac{1}{s} \zeta_K(s)	 - \frac{A}{s-1}  =   \int_0^\infty \e^{-\im t u} \e^{-(\sigma-1)u} \left(\frac{\pi_K(\e^u)}{\e^u} - A\right) \dif u.
\end{equation*}
In any case, by (\ref{korevaar asymptotic formula}) it follows that
\begin{equation*} 
	 \Re \zeta_K ( 1+ \delta + \im t ) = \frac{A \delta}{\delta^2 + t^2} + \Re \phi_K (1 + \delta + \im t),
\end{equation*}
where $\phi_K = \nu(s) \psi_K(s)$, with $\nu(s)$ being a smooth function with fast decay such that $\nu(s) = s$ in the strip $t \in (-2,2)$. This ensures that $\Re \phi_K$ extends to a pseudo-function on $\sigma = 1/2$.
Since convolution operators with pseudo-functions as kernels give compact operators, the sufficiency now follows.

		The converse is more delicate since it is possible for a convolution operator to be compact with a kernel that is not a pseudo-function. For instance, the indicator function $\chi_Y$, where $Y\subset \R$ is unbounded but has finite Lebesgue measure, gives rise to such an operator.
By Theorem \ref{formula theorem} we have the identity
\begin{equation*}
	 \mathcal{Z}_{K,I} - A \mathrm{Id} = \underbrace{\chi_I \mathcal{F}^{-1} \chi_L \mathcal{F} - A \mathrm{Id}}_{(*)} + \Phi_{K,I},
\end{equation*}
for some compact operator $\Phi_{K,I}$. Since the identity operator on $L^2(I)$ can be expressed as $\mathrm{Id} = \chi_I \mathcal{F} \mathcal{F}^{-1}$, it follows from the hypothesis that 
\begin{equation*}
	\sqrt{2\pi}\; \cdot (*) = \chi_I \int_\R (\chi_L - A)\hat{g}(\xi) \e^{\im \xi t} \dif \xi
\end{equation*}
defines a compact operator on $L^2(I)$ for all bounded and symmetric $I \subset \R$. We denote it by $\widetilde{\Psi}$.
It is known that compact operators map sequences that converge weakly to zero to sequences that converge to zero in norm. We use this to show that for all $\delta > 0$,
\begin{equation} \label{modzeta: newproof condition}
	\frac{\abs{L \cap (\xi - \delta, \xi)}}{\delta} - A \rightarrow 0, \quad \text{as} \; \xi \rightarrow \infty.
\end{equation}

Next, let $\epsilon > 0$, write $I = (-T,T)$, for some $T>0$, and for $\xi \in \R$ define the $L^2(-T,T)$ functions 
\begin{equation*}
	 g_\xi(t) = \chi_{(-T,T)} \mathcal{F}^{-1} \{\chi_{(\xi-\delta,\xi)}\}(t) = \sqrt{\frac{2}{\pi}} \e^{\im t (\xi-\frac{\delta}{2})} \frac{\sin (\frac{\delta}{2} t)}{t}.
\end{equation*}
It is clear that for $T>0$ large enough, the real valued functions $\hat{g}_\xi$ approximate the characteristic functions $\chi_{(\xi-\delta,\xi)}$ to an arbitrary degree of accuracy in $L^2(\R)$. This approximation is uniform in $\xi$. In particular, we may choose $T>0$ so that 
\begin{equation*}
	 \frac{1}{2} \delta \leq \norm{g_{\xi}}_{L^2(I)} \leq 2 \delta.
\end{equation*}
Fix some sequence $\abs{\xi_n} \rightarrow \infty$. It follows readily that  the functions $g_{\xi_n}$ converge weakly to zero in $L^2(I)$,
whence $\norm{\Psi g_{\xi_n}} \rightarrow 0$ as $n \rightarrow \infty$. To obtain the connection to the set $L$, we use the dual expression for the norm of $\widetilde{\Psi} g_{\xi_n}$. 
\begin{multline*}
	 \norm{\widetilde{\Psi} g_{\xi_n} }_{L^2(I)} \geq
	 \frac{1}{\norm{g_n}_{L^2(I)}} \Abs{ \int_\R (\chi_L - A) \hat{g}_{\xi_n}(\xi)^2 \dif \xi } \\
	 \geq
	 \underbrace{\frac{1}{2\delta} \Abs{ \int_\R (\chi_L - A) \chi_{(\xi_n-\delta,\xi_n)}(\xi) \dif \xi } }_{(**)}
	 \\  - \underbrace{\frac{1}{2\delta} \Abs{ \int_\R (\chi_L - A) \Big( \hat{g}_{\xi_n}(\xi)^2 - \chi_{(\xi_n-\delta,\xi_n)}(\xi) \Big) \dif \xi }}_{(***)}.
\end{multline*}
It is clear that
\begin{equation*}
	 (**) = \frac{1}{2} \Abs{\frac{\abs{L \cap (\xi_n-\delta, \xi_n)}}{\delta} - A}.
\end{equation*}
Since $\abs{\chi_L - A }\leq 1$ and $\chi_{(\xi_n-\delta,\xi_n)} = \chi_{(\xi_n-\delta,\xi_n)}^2$, we  use the formula $(a^2- b^2) = (a+b)(a-b)$ and the Cauchy-Schwarz inequality to find
\begin{align*}
	 (***) &\leq \frac{1}{2\delta}\norm{\hat{g}_{\xi_n} + \chi_{(\xi_n-\delta,\xi_n)}}_{L^2(I)} \norm{\hat{g}_{\xi_n} - \chi_{(\xi_n-\delta,\xi_n)}}_{L^2(I)} \\ &\leq \frac{3}{2}\norm{\hat{g}_{\xi_n} - \chi_{(\xi_n-\delta,\xi_n)}}_{L^2(I)}.
\end{align*}
By choosing $T>0$ large enough, we have $(***) \leq \epsilon/6$. Hence,
\begin{equation*}
	\Abs{\frac{\abs{L \cap (\xi_n-\delta, \xi_n)}}{\delta} - A}  \leq 2\norm{\widetilde{\Psi} g_{\xi_n}}_{L^2(I)} +  \frac{\epsilon}{2}.
\end{equation*}
Since $\norm{\widetilde{\Psi} g_{\xi_n}}_{L^2(I)} < \epsilon/4$ for large enough $n$, this establishes \eqref{modzeta: newproof condition}.

	To get a contradiction, we 
	assume that $\pi_K(x)/x$ does not tend to the limit $A$. Without loss of generality, we assume that there exists a number $\kappa >0$ such that
	\begin{equation*}
		\limsup_{x \rightarrow \infty} \frac{\pi_K(x)}{x} = A + \kappa.
	\end{equation*}
	This means that for any number $\eta \in (0,1)$ we may find  a strictly increasing sequence of positive numbers $\xi_n$, with arbitrarily large separation, such that $\xi_n \rightarrow \infty$ as $n \rightarrow \infty$ and
	\begin{equation*}
		\frac{\pi_K(\e^{\xi_n})}{\e^{\xi_n}} > A + \eta \kappa \quad \text{for} \; n \in \N.
	\end{equation*}
	Moreover, since the counting function $\pi_K$ changes slowly, there exists a number $\delta_0 > 0$ such that for $n \in \N$ and $\xi \in (\xi_n - \delta_0, \xi_n)$ we have
	\begin{equation*}
		\frac{\pi_K(\e^{\xi})}{\e^{\xi}} - A > \kappa/2.
	\end{equation*}
	Next, for $\xi_n > 2$,
	\begin{align*}
		\abs{L \cap (\xi_n - \delta_0, \xi_n)} &\gtrsim \sum_{\underset{n \in K}{n \in (\e^{\xi_n - \delta_0}, \e^{\xi_n}-1)}} \log \left(1 + \frac{1}{n}\right) \\
		&\gtrsim
		\sum_{\underset{n \in K}{n \in (\e^{\xi_n - \delta_0}, \e^{\xi_n})}} \frac{1}{n} 
		 = \int_{\e^{\xi_n - \delta_0}}^{\e^{\xi_n}} \frac{1}{x} \dif \pi_K(x) \\
		& = \frac{\pi_K(\e^{\xi_n})}{\e^{\xi_n}} - \frac{\pi_K(\e^{\xi_n - \delta_0})}{\e^{\xi_n - \delta_0}} + \int_{\e^{\xi_n - \delta_0}}^{\e^{\xi_n}} \frac{1}{x}\frac{ \pi_K(x)}{x} \dif x.
	\end{align*}
	The last line follows from partial integration, and the implicit constants are absolute. By the properties of $\xi_n$, this implies that
	\begin{equation*}
		 \abs{L \cap ( \xi_n  - \delta_0, \xi_n)} 
		 \gtrsim
		 -(1-\eta) \kappa + \left(A + \frac{\kappa}{2}\right) \delta = A \delta + \left(\eta + \frac{\delta}{2} - 1 \right)\kappa.
	\end{equation*}
	By choosing $\eta  = (4-\delta)/4$, we find that for $\xi_n > 2$,
	\begin{equation*}
		\frac{\abs{L \cap (\xi_n - \delta_0, \xi_n)} }{\delta_0} - A > \frac{\kappa}{4}. 
	\end{equation*}
	This contradicts \eqref{modzeta: newproof condition}.

\end{proof}


\section{Two results in the case when $K$ has arithmetic structure} \label{label: two arithmetic results}

Assume that $K \subset \N$ has  arithmetic structure in the sense that it is the semi-group generated by a subset $Q$ of the prime numbers, which we denote by $\Pri$, i.e. $K$ consists of the integers which are only divisible by primes in $Q$. It follows that we may write 
\begin{equation*}
	 \zeta_K(s) = \prod_{p \in Q} \left( \frac{1}{1 - p^{-s}} \right),
\end{equation*}
In other words, $\zeta_K$ admits an Euler product.
A fundamental fact is that such  $K$ always admit an asymptotic density. We give a proof of this fact, no doubt well-known to specialists, before turning to theorems \ref{arithmetic theorem} and \ref{arithmetic theorem 2}.
\begin{lemma} \label{asymptotic density lemma}
	Let $Q \subset \Pri$ generate the integers $K \subset \N$, and $J$ be the integers generated by the primes not in $Q$. Then 
	\begin{equation*} 
	 	\lim_{x \rightarrow \infty} \frac{\pi_K(x)}{x} = \lim_{\sigma \rightarrow 1^+}\frac{1}{\zeta_J(\sigma)}.
	\end{equation*}
\end{lemma}
\begin{proof}
	This lemma seems to be folklore, indeed for finite $\Pri \backslash Q$ it is readily known that it holds. See for instance  \cite[theorem 3.1]{montgomery_vaughan07}. An immediate consequence is that for infinite $\Pri \backslash Q$, then
	\begin{equation*}
		\limsup_{x \rightarrow \infty} \frac{\pi_K(x)}{x} \leq \lim_{\sigma \rightarrow 1^+} \frac{1}{\zeta_J(\sigma)}.
	\end{equation*}
	In particular, if $\zeta_J(\sigma)$ diverges as $\sigma \rightarrow 1^+$, then $\pi_K(x)/x$ tends to zero. However, the remaining part of the lemma seems to be more difficult, and no analytic proof, or indication thereof, seems to be readily available in the literature. Therefore we show how one follows  from the Ikehara-Korevaar theorem above.

	Assume that $\zeta_J(1) < \infty$ and recall that $\zeta(s) = (s-1)^{-1} + \psi(s)$ for some entire function $\psi$.
	By Lemma 1, it suffices to show that the following function coincides with a pseudo-function on finite intervals along the abscissa $\sigma = 1$.
	\begin{eqnarray*}
		\frac{ \zeta_K(s)}{s} - \frac{1}{\zeta_J(1)} \frac{1}{s-1}
		&=&
		\frac{\zeta(s)}{s\zeta_J(s)}  - \frac{1}{\zeta_J(1)} \frac{1}{s-1} \\
		&=&
		\frac{1}{s-1} \left( \frac{1}{s\zeta_J(s)} - \frac{1}{\zeta_J(1)} \right) + \frac{\psi(s)}{s\zeta_J(s)}.
	\end{eqnarray*}
	Since $\zeta_J(1)<\infty$ it is not hard to use the Euler product formula to see that $\zeta_J(1 + \im t)$  is bounded above and below in absolute value for all $\R$. This means that the last term coincides with a pseudo-function on finite intervals along the abscissa $\sigma = 1$. Hence, the same is true for the left-hand side if and only if it holds true for the first term on the right-hand side. It is readily seen that this function extends to a pseudo-function on $\sigma=1$ if and only if the same is true for 
	\begin{equation} \label{its}
		 \frac{1}{s-1} \left( \frac{\zeta_J(s)}{s}  - \zeta_J(1) \right).
	\end{equation}
	We calculate its distributional Fourier transform. Let $\phi$ be a test function.
	Since we may write
	\begin{equation*}
	 	\frac{\zeta_J(s)}{s} = \frac{1}{s} \int_{1}^\infty x^{-s} \dif \pi_J(x) =\int_{0}^\infty \frac{\pi_J(\e^u)}{\e^u} \e^{-(\sigma-1) - \im t u}\dif u,
	\end{equation*}
	 it follows that
\begin{multline*}
	 \lim_{\delta \rightarrow 0} \int_\R \hat{\phi}(t) \frac{1}{\delta + \im t}\left(\frac{\zeta_J(1+\delta + \im t)}{1+\delta +\im t}  - \zeta_J(1) \right) \dif t \\
	= \lim_{\delta \rightarrow 0} \int_\R \frac{1}{\delta + \im t}\hat{\phi}(t)\int_0^\infty g(u) (\e^{-\delta u - \im u t}- 1)\dif u \dif t,
\end{multline*}
where $g(u) = \pi_J(\e^u) e^{-u}$.
Using the smoothness of $\phi$, we change the order of integration,
\begin{multline*}
	 	\lim_{\delta \rightarrow 0}\int_0^\infty g(u) \int_\R \hat{\phi}(t) \frac{\e^{-\delta u - \im u t}- 1}{\delta + \im t}\dif t \dif u
		\\=
	 	 \int_0^\infty g(u) \lim_{\delta \rightarrow 0}\int_\R \hat{\phi}(t) \frac{\e^{-\delta u - \im u t}- 1}{\delta + \im t}\dif t \dif u
		 \\=
		 \int_0^\infty g(u) \Phi(u) \dif u,
\end{multline*}
where $\Phi'(u) = - 2\pi \phi(u)$ and $\Phi(0)= 0$.
This means that
\begin{equation*}
	 \Phi(u) = -2\pi \int_0^u \phi(x) \dif x \quad \text{for} \; u \geq 0.
\end{equation*}
So,
\begin{multline*}
	 \lim_{\delta \rightarrow 0} \int_\R \hat{\phi}(t) \frac{1}{\delta + \im t}\left(\frac{\zeta_J(1+\delta + \im t)}{1+\delta +\im t}  - \zeta_J(1)\right) \dif t 
	 \\= 
	 -2\pi \int_0^\infty g(u)  \int_0^u \phi(x) \dif x \dif u
	 \\ =
	- 2 \pi \int_\R \phi(x) \chi_{(0,\infty)}(x)\int_{x}^\infty g(u) \dif u \dif x.
\end{multline*}
Since $g(u)$ is integrable, this implies that 
\begin{equation*}
	 \chi_{(0,\infty)}(x)\int_{x}^\infty g(u) \dif u
\end{equation*}
decays as $\abs{x} \rightarrow \infty$ and so the function \eqref{its} extends to a pseudo-function on the abscissa $\sigma=1$.
\end{proof}
Hence, if  $J$ denotes the integers generated by the primes not in $Q$, then the condition (\ref{asymptotic density}) always holds with $A = \lim_{\sigma \rightarrow 1^+}\zeta_J^{-1}(\sigma)$.
By the Euler product representation of $\zeta_J$
it is seen that $\zeta_J(1) < \infty$ if and only if 
	\begin{equation}
		 \sum_{p \in \Pri \backslash Q} p^{-1} < \infty.
			 \label{finite prime sum condition}
	\end{equation}
Under these conditions (\ref{panejah's condition}) is equivalent to (\ref{finite prime sum condition}). This means that we get the following simpler form of theorems \ref{formula theorem}  to \ref{identity theorem}.
\begin{theorem}  \label{arithmetic theorem}
	Let $I \subset \R$ be a bounded symmetric interval, $Q \subset \Pri$ generate the integers $K$, and $J$ be the integers generated by the primes not in $Q$. Then 
	\begin{equation} \label{tjosan}
	 	\mathcal{Z}_{K,I} = \zeta_J^{-1}(1) \mathrm{Id} + \Psi_{K,I}, 
	\end{equation}
	for a compact operator $\Psi_{K,I}$.
	Moreover, the operator $\mathcal{Z}_{K,I}$ is bounded below on $L^2(I)$ if and only if
	\begin{equation*} 
		\sum_{p \in \Pri \backslash Q} \frac{1}{p} < 	\infty.
	\end{equation*}
\end{theorem}
\begin{proof}
	By Lemma \ref{asymptotic density lemma} the limit
	\begin{equation*}
		\lim_{x \rightarrow \infty}  \frac{\pi_K(x)}{x}  = A
	\end{equation*}
	always holds with $A = \zeta_J^{-1}(\sigma)$. With this, Theorem \ref{identity theorem} implies the formula for $\mathcal{Z}_{K,I}$.
	
	Finally, Theorem \ref{lower bound theorem} says that $\mathcal{Z}_{K,I}$ is bounded below if and only if $A>0$. By considering the Euler product of $\zeta_J(s)$ it follows that $\zeta_J(1) < +\infty$ is exactly the condition of the theorem.
\end{proof}



Since (\ref{finite prime sum condition}) is equivalent to (\ref{korevaar asymptotic formula}) with $A >0$, the formula \eqref{tjosan} may be seen as a direct consequence of the Ikehara-Korevaar theorem.
However, more can be said in relation to the formula (\ref{korevaar asymptotic formula}).  Note that $f \in L^p_\loc$ if $f \in L^p(E)$ for any compact $E \subset \R$.
\begin{theorem} \label{arithmetic theorem 2}
	Let $Q \subset \Pri$ generate the integers $K$, and $J$ be the integers generated by the primes not in $Q$, and 
	assume that \eqref{finite prime sum condition} holds.
	Then 
	\begin{equation*}
	 	\psi_K(s) := \frac{1}{s} \zeta_K(s) - \frac{\zeta_J^{-1}(1)}{s-1}
	\end{equation*}
	extends, in the sense of distributions, to a function
	 in $L^1_{\mathrm{loc}}$ on the abscissa $\sigma=1$ if and only if
	\begin{equation*}
		\sum_{p \in \Pri \backslash Q} \frac{\log \log p}{p} < \infty.
	\end{equation*}
	For $q>1$, the extension is in  $L^q_{\mathrm{loc}}$ on the abscissa $\sigma=1$ if
	\begin{equation*}
		\sum_{p \in \Pri \backslash Q} \frac{\log^{1/q'} p}{p} < \infty,
	\end{equation*}
	where $q'>1$ is the real number satisfying $q^{-1} + q'^{-1} = 1$. Conversely,
	if 
	\begin{equation*}
		\sum_{p \in \Pri \backslash Q} \frac{\log^{1/q'} p}{p} = \infty,
	\end{equation*}
	then the extension of $\psi$ on $\sigma=1$ is not in $L^{r}_{\mathrm{loc}}$ for $r>q$.
\end{theorem}
Before the proof of the theorem, which follows an argument similar to that of Lemma \ref{asymptotic density lemma}, we give a lemma which ties together summability conditions on subsets of integers and their generating prime numbers. It follows easily by using the measure calculus described by P.~Malliavin in \cite{malliavin61}, however we provide a more elementary argument for the readers convenience.
\begin{lemma} \label{malliavin lemma}
	Let $f : \N \rightarrow \R_+$ satisfy $f(nm) \leq f(n) + f(m)$, $f(1) = 0$ and $f(n) \geq 1$ for $n$ big enough. 
	If the primes $P$ generate the integers $J$ then
	\begin{equation*}
		\sum_{n \in J} \frac{f(n)}{n} < \infty \quad \iff \quad \sum_{p \in P} \frac{f(p)}{p} < \infty.
	\end{equation*}
\end{lemma}
\begin{proof}
	One way to prove this is, for $\sigma > 1$, to establish the inequality
	\begin{equation*}
		\sum_{n\in J} f(n) n^{-\sigma}	 \lesssim \left( \sum_{p \in P} f(p) p^{-\sigma} \right)
		\e^{\sum_p p^{-\sigma}},
	\end{equation*}
	and then conclude by the monotone convergence theorem.
	To achieve this we study the linear map $D_f : \sum a_n n^{-\sigma} \rightarrow \sum a_n f(n) n^{-\sigma}$. It is not hard to show that the abscissa of absolute convergence is invariant under $D_f$. Moreover, for Dirichlet series $F,G$ with positive coefficients, it holds that $D_f(FG)(\sigma) \leq D_f(F) G(\sigma) + F D_f(G)(\sigma)$. We use this on the identity
	\begin{equation*}
		\sum_{n\in J } n^{-\sigma} = \e^{\sum_{p \in P}p^{-\sigma}} R(\sigma),
	\end{equation*}
	where the function
	\begin{equation*}
		R(\sigma) = \e^{- \sum_{p \in P} \log(1-p^{-\sigma}) - \sum_{p \in P} p^{-\sigma} } 
	\end{equation*}
	is given by a Dirichlet series that converges absolutely for $\sigma > 1/2$. 
	Here we used the Euler product formula for the function $\zeta_J$.
	The desired inequality is now seen to hold since
	\begin{equation*}
		D_f(  \e^{\sum_{p \in P}p^{-\sigma}} ) \leq \left(\sum_{p \in P} f(p) p^{-\sigma} \right)\e^{\sum_{p \in P} p^{-\sigma}}.
	\end{equation*}
\end{proof}
	\begin{proof}[Proof of Theorem \ref{arithmetic theorem 2}]
	Both the formula
	\begin{equation*}
		\mathcal{Z}_{K,I} = \zeta_J^{-1}(1) \mathrm{Id} + \Psi_{K,I}, 
	\end{equation*}
	where $\Psi_{K,I}$ is a compact operator,
	and the statement that $\psi_K$ extends to a pseudo-function on  $\sigma=1$, 
	follow immediately from Theorem \ref{identity theorem}.

	As in the hypothesis, assume that $\zeta^{-1}_J(1) > 0$. By the factorisation $\zeta(s) = \zeta_K(s) \zeta_J(s)$ and the
	formula (\ref{asymptotic formula}) for the Riemann zeta function we have the identity
	\begin{align*}
		\psi_K(s) &=  \frac{\zeta_K(s)}{s} - \frac{1}{\zeta_J(1)} \frac{1}{s-1}\\
			&= \frac{1}{s-1} \left( \frac{1}{s\zeta_J(s)} - \frac{1}{\zeta_J(1)} \right)
			+ \frac{\psi(s)}{s}.
	\end{align*}
	Under our assumption, it follows from the Euler product formula
	that $\zeta_J(1 + \im t)$ is a continuous function bounded away from zero. Therefore
	\begin{equation*} 
		\frac{1}{t} \left( \frac{1}{\zeta_J(1+ \im t)} - \frac{1}{\zeta_J(1)} \right) \in L^q_{\loc}(\R)
		\iff
		\frac{ \zeta_J(1+ \im t) - \zeta_J(1) }{t} \in L^q_\loc(\R).
	\end{equation*}
	Let $q' \in [1,\infty]$ be such that $q^{-1} + q'^{-1} = 1$. Hence, by duality 
	\begin{equation*}
		\Norm{\frac{\zeta_J(1+ \im t) - \zeta_J(1)}{t} }_{L^q(I)}
		=
		\sup_{\phi \in L^{q'}(I)} \Abs{ \int_I \phi(t) \frac{\zeta_J(1 + \im t) - \zeta_J(1)}{t} \dif t}.
	\end{equation*}
	A brief calculation, where we use Fubini's theorem twice along with the properties of the counting measure $\dif \pi_J$, 
	\begin{align*}
		\int_I \phi(t)  \frac{\zeta_J(1 + \im t) - \zeta_J(1)}{t} \dif t
		&=
		\int_1^\infty \int_I \phi(t) \frac{x^{-\im t} - 1 }{t} \dif t \frac{\dif \pi_J(x)}{x} \\
		&= 
		\sqrt{2\pi } \int_1^\infty \int_0^{\log x} \hat{\phi}(\xi) \dif \xi \frac{\pi_J(x)}{x}\\
		&=
		\sqrt{2\pi} \sum_{n \in J} \frac{1}{n} \underbrace{\int_0^{\log n} \hat{\phi} (\xi) \dif \xi}_{(*)}.
	\end{align*}
	We proceed to estimate $(*)$ for $\phi$ in $L^{q'}(I)$. First we let $q=1$. Then $q' = \infty$. Recall that $I$ is a bounded and symmetric interval, hence $I = (-T,T)$ for some $T>0$. By Fubini's theorem and a change of variables,
	\begin{align*}
		\Abs{\int_0^u \hat{\phi}(\xi) \dif \xi} 
		&=
		\frac{1}{\sqrt{2\pi}} \Abs{\int_{-T}^T \phi(t) \frac{\e^{-\im t u} - 1}{t} \dif t}
		\\
		&=
		\frac{1}{\sqrt{2\pi}} \Abs{\int_{-uT}^{uT} \phi\left(\frac{t}{u}\right) \frac{\e^{-\im t} - 1}{t} \dif t}
		\leq
		\frac{\norm{\phi}_{L^\infty(I)}}{\sqrt{2\pi}} \int_{-uT}^{uT} \frac{\abs{\e^{-\im t} -1}}{t} \dif t.
	\end{align*}
	The integral in the last expression is clearly $\Bigoh{\log u}$.
	  It now follows that
	\begin{equation*}
		\Norm{\frac{\zeta_J(1+ \im t) - \zeta_J(1)}{t} }_{L^1(I)}
		 \lesssim \sum_{n \in J}  \frac{1}{n} \log \log n.
	\end{equation*}	
	By Lemma \ref{malliavin lemma}, 
	\begin{equation*} 
	 	\sum_{n \in J} \frac{\log \log n}{n} < \infty \iff 
		\sum_{p \in \Pri \backslash Q} \frac{\log \log p}{p} < \infty.
	\end{equation*}		
	This proves the sufficiency for $q = 1$. As for the necessity, assume that \break
	$\sum_{p \in \Pri \backslash Q} p^{-1}\log \log p = \infty$ and set $\phi = \chi_{I}(x)$. With this choice
	\begin{equation*}
		\int_0^u \hat{\phi}(\xi)\dif \xi = \frac{1}{\sqrt{2\pi}} \int_{-T}^{T} \frac{\e^{-\im t u}-1}{\im t} \dif t = \int_{-uT}^{uT} \frac{\e^{-\im t}-1}{\im t} \dif t.
	\end{equation*}
	The result now follows since it is clear that for fixed $u >0$ it holds that, 
	\begin{equation*}
		\int_{0}^{u} \frac{\e^{-\im t}-1}{t} \dif t < \infty,
	\end{equation*}
	while as $u$ grows we have
	\begin{equation*}
		\int_{1}^u \frac{\dif t}{t}= \log u. 
	\end{equation*}
	
	For $q > 1$, the necessity is proved by using the fact that for $\phi \in L^{q'}(I)$ we have
	\begin{equation*}
		\Abs{ \int_0^u \hat{\phi}(\xi) \dif \xi} \lesssim \norm{\phi}_{q'} u^{1/q'}.
	\end{equation*}
	The sufficiency follows since for all $r > q'$ there exists $\phi \in L^{q'}(I)$ for which
	\begin{equation*}
		\int_0^u \hat{\phi}(\xi) \dif \xi = x^{1/r} + \Bigoh{1}.  
	\end{equation*}
 We leave the details to the reader.
	\end{proof}

\section{Remarks on the relation to the prime number theorem} \label{section: prime numbers}

In this section, we present our final result on the operator $\mathcal{Z}_{K,I}$. 
Let $R = (r_i)$ be an increasing sequence of real numbers greater than one and let $N$ be the multiplicative semi-group it generates. We say that $R$ is the Beurling prime numbers for the Beurling integers $N$. This point of view leads to a generalised type of number theory, initiated by Arne Beurling in \cite{beurling37}. The focus of the theory is to investigate how the asymptotic structure of $R$ relates to that of $N$. For a survey, see \cite{hilberdink_lapidus06}.
In our case, $Q$ corresponds to the Beurling primes and $K$ to the Beurling integers. 
We say that the prime number theorem holds for $Q$ if
\begin{equation*} 
	\pi_Q(x) \sim \frac{x}{\log x}.
\end{equation*}
The following is now true.
\begin{theorem} \label{beurling theorem}
	Let $Q \subset \Pri$ generate $K \subset \N$. Then the prime number theorem for $Q$ neither implies nor is implied by the lower boundedness of the operator $\mathcal{Z}_{K,I}$.
\end{theorem}

To prove the theorem, we need a lemma.
Note that the symbol $f(x) \sim g(x)$ is taken to mean $f(x)/g(x) \rightarrow 1$ as $x \rightarrow \infty$. 
\begin{lemma} \label{ pnt lemma and proposition}
	Let $Q \subset \Pri$ generate $K \subset \N$ and let $J$ denote the integers generated by the primes not in $Q$. Then the prime number theorem holds for the set $K$ if and only if
	\begin{equation} \label{ prime theorem condition}
		\sum_{p \in \Pri\backslash Q \cap (\delta x, x)} \frac{\log p}{p} = \Littleoh{1}, \quad \text{for all} \; \delta \in (0,1).
\end{equation}
\end{lemma}
\begin{proof}
	Let $P = \Pri \backslash Q$. It is clear that the prime number theorem holds for $K$ if and only if $\pi_P(x) = \littleoh(x/\log x)$. Moreover, it is readily seen that
\begin{equation*}
	 \frac{\log x}{ x} (\pi_P( x) - \pi_P(\delta x)) \leq \sum_{p \in (\delta x,x)} \frac{\log p}{p} \leq \frac{\log \delta x}{ \delta x} (\pi_P( x) - \pi_P( \delta x)).
\end{equation*}
So we have to show that $\pi_P(x) = \littleoh(x/\log x)$ is equivalent to the statement that for all $\delta >0$ it holds that $\pi_P(x) - \pi_P(\delta x) = \littleoh(x/\log x)$. One direction is immediate. 
 For the other, assume that $\pi_P(x) - \pi_P(\delta x) = \Littleoh{x/\log x}$. Rewrite this assumption in the form
	\begin{equation*}
	 	\pi_P(x) \frac{\log x}{x} = \delta \left( \pi_P(\delta x) \frac{\log \delta x}{\delta x} \right)\frac{\log x}{\log \delta x} + \Littleoh{1}.
	\end{equation*}
	Hence, for all $\delta>0$, we have
	\begin{equation*}
		\limsup_{x \rightarrow \infty} \pi_P(x) \frac{\log x}{x} < \delta.
	\end{equation*}
\end{proof}

	\begin{proof}[Proof of Theorem \ref{beurling theorem}]
	Recall that by Theorem \ref{arithmetic theorem}, the lower boundedness of the operator $\mathcal{Z}_{K,I}$ is equivalent to the condition
	\begin{equation} \label{ betingelse 2}
		\sum_{p \in \Pri \backslash Q} \frac{1}{p} < \infty.
	\end{equation}
	Moreover, an immediate consequence of  Lemma \ref{asymptotic density lemma} is that this is equivalent to the condition
	\begin{equation} \label{ betingelse 1}
		\liminf_{x \rightarrow \infty }\frac{\pi_K(x)}{x} > 0.
	\end{equation}

	First we seek a set of primes $Q$ for which Panejah's condition holds but the prime number theorem does not.
	This part of the theorem follows by comparing the condition \eqref{ betingelse 1} to the condition \eqref{ prime theorem condition} of Lemma \ref{ pnt lemma and proposition} in combination with 
	a variant of Merten's formula (see e.g. \cite{montgomery_vaughan07}[p. 50]):
	\begin{equation} \label{ the formula}
		\sum_{p \leq x} \frac{\log p}{p}  = \log x + \Bigoh{1}. 
	\end{equation}
	One the one hand, \eqref{ the formula} implies that for
	$\delta > 0$ small enough, then
	\begin{equation*}
		\liminf_{x \rightarrow \infty} \sum_{\delta x \leq p \leq x} \frac{\log p}{p}  > 0.
	\end{equation*}
	We choose a sequence $\seq{x_n}_{n\in \N}$ which realises this condition and for which the intervals $(\delta x_n, x_n)$ do not overlap. On the other hand, \eqref{ the formula} implies that
	\begin{equation*}
		\sum_{\delta x \leq p \leq x} \frac{1}{p} \lesssim \frac{ 1}{\log \delta x}.
	\end{equation*}
	Choose a sub-sequence of $\seq{x_{n_k}}$ for which $\sum_{k} (\log x_{n_k})^{-1}< \infty$. Let $P = \Pri \cap \left( \cup(\delta x_{n_k}, x_{n_k})\right)$, and set $Q  = \Pri \backslash P$. This set does the job.
	
	Next, we seek a set $Q$ for which the prime number theorem holds, but Panejah's condition fails. Consider the consecutive intervals $I_k= (2^k,2^{k+1})$. In each interval, choose essentially the first $2^{k}/(k \log k)$ prime numbers. This is seen to be exactly possible for large $k$ using the fact that the $n$'th prime $p_n \sim n \log n$. Denote the set of primes chosen in this way from the interval $I_k$ by $P_k$. Set $P = \cup P_k$ and let $Q = \Pri \backslash P$. 
	It now follows that the condition \eqref{ betingelse 2} does not hold, since
	\begin{equation*}
		\sum_{p \in P} \frac{1}{p} \gtrsim \sum_{k \in \N} \frac{1}{k \log k} = \infty. 
	\end{equation*}
	To see that the prime number theorem for $Q$ holds, we let $\delta \in (0,1)$ and
	readily 
	check that for $x > x_\delta$ we have
	\begin{equation*}
		\sum_{p \in P \cap (\delta x,x)} \frac{\log p}{p} \lesssim \frac{\log x}{ x} \frac{x}{\log x \log \log x } = \frac{1}{\log x}.
	\end{equation*}
	Hence the condition \eqref{ prime theorem condition} of Lemma \ref{ pnt lemma and proposition} holds.
\end{proof}


\section{Concluding remarks} \label{concluding remarks}

The connection between the operator $\mathcal{Z}_{K,I}$ and the frame \eqref{the frame} was essentially observed in the paper \cite{olsen_saksman09}. There the operator $\mathcal{Z}_{\N,I}$ was used to study the Dirichlet-Hardy space
\begin{equation*}
	 \mathscr{H}^2 = \Set{ \sum_{n \in \N} a_n n^{-s} \; : \; \sum_{n \in \N} \abs{a_n}^2 < + \infty }.
\end{equation*}
By the Cauchy-Schwarz inequality, the functions in this space are analytic for $\sigma > 1/2$. It has been shown \cite{montgomery94, hls97}  that for a bounded interval $I$ there exists a constant $C>0$, only depending on the length of the interval $I$, such that for every $F \in \mathscr{H}^2$,
\begin{equation*}
	  \int_{I} \Abs{F\left( \frac{1}{2} + \im t\right)}^2 \dif t 
	 \leq C \norm{F}^2_{\mathscr{H}^2}.
\end{equation*}
This implies that for $F \in \Hp^2$ then $F(s)/s$ is in the classical Hardy space $H^2$ on the half-plane $\sigma > 1/2$. 
In particular, it follows that functions in $\Hp^2$ have non-tangential boundary values almost 
everywhere on the abscissa $\sigma=1/2$. This gives meaning to the notation $F(1/2 + \im t)$ for $F \in \Hp^2$.
A special case of the main result of \cite{olsen_saksman09} is now stated as follows.
\begin{theoa}[Olsen and Saksman 2009]
	Let $I$ be some bounded interval in $\R$ and $v \in L^2(\R)$. Then there exists a function $F \in \Hp^2$ 
	such that $\Re F(1/2 + \im t) = v(t)$ almost everywhere in $L^2(I)$.
\end{theoa}
This result may be reformulated as saying that $\mathscr{G}_\N$, as defined in \eqref{the frame}, forms a frame for $L^2(I)$.
As a consequence, the Corollary \ref{the corollary} implies that we can replace $\Hp^2$ by any of the subspaces
\begin{equation*}
	 \mathscr{H}^2_K = \Set{ \sum_{n \in K} a_n n^{-s} \; : \; \sum_{n \in K} \abs{a_n}^2 < + \infty }
\end{equation*}
for which $K$ satisfies condition $(c)$ of Theorem \ref{lower bound theorem}.
For more on the emerging theory of the space $\mathscr{H}^2$, see \cite{hls97, gordon_hedenmalm99, bayart02, konyagin_queffelec02, hedenmalm_saksman03, bayart03, mccarthy04, helson05a,  queffelec07, saksman_seip08, seip08}.

Finally we mention that the spaces $\Hp^p$ were defined for arbitrary $p>0$ by F.~Bayart in \cite{bayart02}. By an idea of H.~Bohr, they are defined to be the Dirichlet series for which the coefficients are Fourier coefficients of functions in the Hardy space $H^p(\T^\infty)$, where $\T^\infty = \set{(z_1, z_2, \ldots) : z_j \in \T})$ is equipped with the product topology. For $p=2$ this definition coincides with the definition of $\Hp^2$ given above. The behaviour of functions in these spaces is for the most part unknown.

\section*{Acknowledgements}

This paper forms a part of the author's PhD dissertation under the advice of Professor Kristian Seip. The research done for this paper was in large part done during a stay at Washington University in St. Louis, Missouri, USA. The author is grateful to Professor John E. McCarthy for the invitation and stimulating discussions.

\bibliographystyle{amsalpha}
\bibliography{../../../_bibliotek/bibliotek}

\end{document}